\documentclass[11pt]{amsart}

\usepackage[utf8]{inputenc} 
\usepackage[T1]{fontenc} 
\usepackage[english]{babel} 
\usepackage{mathrsfs} 

\DeclareMathOperator{\mre}{Re}

\theoremstyle{plain} 
\newtheorem{theorem}{Theorem} 
\newtheorem{lemma}[theorem]{Lemma} 
\newtheorem{corollary}[theorem]{Corollary} 

\theoremstyle{remark} 
\newtheorem*{remark}{Remark} 
\newtheorem*{problem}{Problem}

\begin{document} 
\title[High pseudomoments of the Riemann zeta function]{High pseudomoments of \\ the Riemann zeta function} 
\date{\today} 

\author{Ole Fredrik Brevig} 
\address{Department of Mathematical Sciences, Norwegian University of Science and Technology (NTNU), NO-7491 Trondheim, Norway} 
\email{ole.brevig@math.ntnu.no}

\author{Winston Heap} 
\address{Department of Mathematics, University College London, 25 Gordon Street, London WC1H.} 
\email{winstonheap@gmail.com}

\begin{abstract}
	The pseudomoments of the Riemann zeta function, denoted $\mathcal{M}_k(N)$, are defined as the $2k$th integral moments of the $N$th partial sum of $\zeta(s)$ on the critical line. We improve the upper and lower bounds for the constants in the estimate $\mathcal{M}_k(N) \asymp_k (\log{N})^{k^2}$ as $N\to\infty$ for fixed $k\geq1$, thereby determining the two first terms of the asymptotic expansion. We also investigate uniform ranges of $k$ where this improved estimate holds and when $\mathcal{M}_k(N)$ may be lower bounded by the $2k$th power of the $L^\infty$ norm of the $N$th partial sum of $\zeta(s)$ on the critical line. 
\end{abstract}


\maketitle

% INTRODUCTION
\section{Introduction} Let $k$ be a positive real number, and let $\zeta(s) = \sum_{n=1}^\infty n^{-s}$ denote the Riemann zeta function. Over the past century, the \emph{moments} 
\begin{equation}\label{eq:moments} 
	M_k(T) = \frac{1}{T}\int_0^T |\zeta(1/2+it)|^{2k}\,dt 
\end{equation}
have received considerable attention. The cases $k=1$ and $k=2$ were computed by Hardy and Littlewood \cite{HL16} and Ingham \cite{Ingham26}, respectively, who found that as $T\to\infty$,
\[M_1(T) \sim \log{T} \qquad\text{and}\qquad M_2(T) \sim \frac{1}{2\pi^2} (\log{T})^4.\]
Keating and Snaith \cite{KS00} conjectured that 
\begin{equation}\label{eq:ksconj} 
	\lim_{T \to \infty} \frac{M_k(T)}{(\log{T})^{k^2}} = a(k)g(k) 
\end{equation}
for every fixed positive real number $k$. Here $a(k)$ denotes the Euler product 
\begin{equation}\label{eq:ak} 
	a(k) = \prod_{p} \left(1-\frac{1}{p}\right)^{k^2}\sum_{j=0}^\infty \frac{d^2_k(p^j)}{p^j}, \qquad d_k(p^j) = \binom{j+k-1}{j}, 
\end{equation}
and $g(k)$ is a specific function arising from random matrix theory.

One motivation for studying the moments \eqref{eq:moments} is their connection to large values of the Riemann zeta function on the critical line. Set 
\begin{equation}\label{eq:lindelof} 
	M(T) = \max_{0\leq t \leq T} |\zeta(1/2+it)|. 
\end{equation}
The Lindel\"of hypothesis states that $M(T) \ll_\varepsilon T^\varepsilon$ for every $\varepsilon>0$, and it follows from the Riemann hypothesis that $\log M(T) \ll \log{T}/\log\log{T}$. Clearly, \eqref{eq:lindelof} can be computed as the following limit of the moments 
\begin{equation}\label{eq:lindelimit} 
	M(T) = \lim_{k\to\infty} [M_k(T)]^{1/(2k)}. 
\end{equation}
Farmer, Gonek and Hughes \cite{FGH07} demonstrated that the conjecture \eqref{eq:ksconj} cannot hold uniformly for $k\geq C\sqrt{\log{T}/\log\log{T}}$ for some specific constant $C$. However, by inserting the largest possible $k$ into \eqref{eq:lindelimit}, they conjectured that
\[M(T) = \exp\left(\left(\frac{1}{\sqrt{2}}+o(1)\right)\sqrt{\log{T}\log\log{T}}\right).\]
This conjecture was also derived by other methods.

In the present paper, we investigate similar problems for \emph{pseudomoments} of the Riemann zeta function. The pseudomoments exhibit some of the same properties as the moments \eqref{eq:moments}, while being comparably tractable in many cases. For a Dirichlet series $f(s) = \sum_{n=1}^\infty a_n n^{-s}$, its $N$th partial sum is
\[S_Nf(s) = \sum_{n=1}^N a_n n^{-s}.\]
The $k$th pseudomoment of the Riemann zeta function is the limit 
\begin{equation}\label{eq:pseudomoments} 
	\mathcal{M}_k(N) = \lim_{T\to \infty} \frac{1}{T}\int_0^T |S_N \zeta(1/2+it)|^{2k} \, dt. 
\end{equation}
Expanding the integrand and computing, we get that $\mathcal{M}_1(N)\sim \log{N}$. The study of pseudomoments was initiated by Conrey and Gamburd \cite{CG06}, who demonstrated that if $k$ is a fixed positive integer, then 
\begin{equation}\label{eq:conreygamburd} 
	\lim_{N\to \infty} \frac{\mathcal{M}_k(N)}{(\log{N})^{k^2}} = a(k) \gamma(k). 
\end{equation}
Here $a(k)$ is the Euler product \eqref{eq:ak} and $\gamma(k)$ is the volume of the convex polytope 
\begin{equation}\label{eq:Pk} 
	\mathcal{P}_k = \left\{(x_{ij}) \in \mathbb{R}^{k^2} \,:\, x_{ij}\geq0, \quad \sum_{i=1}^k x_{ij} \leq 1, \quad \sum_{j=1}^k x_{ij} \leq 1\right\}. 
\end{equation}
In particular, setting $k=2$ in \eqref{eq:conreygamburd} gives that $\mathcal{M}_2(N) \sim (\log{N})^4 / \pi^2$.

Bondarenko, Heap and Seip \cite{BHS15} investigated \eqref{eq:pseudomoments} for continuous $k$. A special case of their main result implies that for every fixed real number $k>1/2$, it holds that 
\begin{equation}\label{eq:Mkasymp} 
	\mathcal{M}_k(N) \asymp_k (\log{N})^{k^2} 
\end{equation}
as $N\to\infty$. The situation for $0<k\leq1/2$ is less clear, we refer again to \cite{BHS15} and to the recent work of the second named author \cite{Heap17}. 

The estimates for the implied constant in the upper bound of \eqref{eq:Mkasymp} were recently substantially improved in \cite{BBSSZ}. Setting\footnote{Although we do not know that the limit \eqref{eq:Cklimit} exists when $k$ is not an integer (in which case the existence follows \eqref{eq:conreygamburd}), we will slightly abuse the notation and assume that $\mathcal{C}(k)$ exists. This can be justified by noting that the upper and lower bounds are actually upper and lower bounds for $\limsup$ and $\liminf$ of the limit \eqref{eq:Cklimit}, respectively.} 
\begin{equation}\label{eq:Cklimit} 
	\mathcal{C}(k) = \lim_{N\to\infty} \frac{M_k(N)}{(\log{N})^{k^2}}, 
\end{equation}
the previously known lower and upper bounds combine to
\[\exp(-(2+o(1))k^2\log{k})\leq \mathcal{C}(k)\leq \exp(-(1+o(1))k^2\log{k}),\]
as $k\to\infty$. The main goal of the present paper is to sharpen this estimate and to obtain a uniform range of $k$ where this improved estimate holds.
\begin{theorem}\label{thm:asymptotics} 
	Uniformly for $2 \leq k \leq c \sqrt{\log\log{N}}$ it holds that 
	\[\frac{\mathcal{M}_k(N)}{(\log{N})^{k^2}} \geq \exp\left(-k^2\log{k}-k^2\log\log{k} + O(k^2)\right).\]
	Uniformly for $2 \leq k \leq C_1 \log{N}/\log\log{N}$ it holds that 
	\[\frac{\mathcal{M}_k(N)}{(\log{N})^{k^2}} \leq \exp\left(-k^2\log{k}-k^2\log\log{k} + O(k^2)\right).\]
	The upper bound does not hold for $k = C_2 \log{N}/\log\log{N}$. 
\end{theorem}
In Theorem~\ref{thm:asymptotics} and throughout the paper we will let constants such as $C_1$ and $C_2$ be sufficiently large or small depending on the context in which they are used. \begin{remark}
	The statements of Theorem~\ref{thm:asymptotics} also hold for $k\geq 1/2+\delta$ for any $\delta>0$ in view of the results from \cite{BHS15}. However, the main interest of the asymptotic estimates is large $N$ and large $k$. We will therefore generally assume that $N$ and $k$ are large enough for various logarithms to be positive.
\end{remark}

Theorem~\ref{thm:asymptotics} is in agreement with the asymptotic behaviour of the constants appearing in the Keating--Snaith conjecture \eqref{eq:ksconj}. We also note that Harper \cite{Harper18} has very recently obtained similar results for the analogous moments on the line $\sigma=0$, 
\begin{equation}\label{eq:sigma0} 
	\lim_{T\to\infty} \frac{1}{T}\int_0^T |S_N \zeta(it)|^{2k}\,dt. 
\end{equation}
It should be made clear that the techniques used in the proof of Theorem~\ref{thm:asymptotics} are different from those of \cite{Harper18}. It is in fact easy to check that our proof does not work for $\sigma<1/2$. However, our techniques are quite flexible on the critical line and it is possible to extend Theorem~\ref{thm:asymptotics} to moments of other Dirichlet polynomials considered in \cite{BHS15}.

The key ingredient in the proof of Theorem~\ref{thm:asymptotics} is Weissler's inequality for Dirichlet polynomials (see \cite{Bayart02,Weissler80}). This inequality allows us to estimate non-integer pseudomoments through estimates for integer pseudomoments of homogeneous completely multiplicative twists. We will estimate these twisted integer moments by the theory for certain multiple Dirichlet series developed in \cite{BNP08,CG06,HL16} for the lower bound and using Rankin's trick for the upper bound. 

Indeed, for a non-negative real number $\varrho$ define 
\begin{equation}\label{eq:tmoment} 
	\mathcal{M}_{k,\varrho}(N) = \lim_{T\to\infty} \frac{1}{T}\int_0^T \left|\sum_{n=1}^N \frac{\varrho^{\Omega(n)}}{n^{1/2+it}}\right|^{2k}\,dt
\end{equation}
where $\Omega(n)$ denotes the number of prime factors of $n$, counting multiplicities. Then Weissler's inequality (see Section~\ref{sec:asymptotic}) gives that
\[\left(\mathcal{M}_{\lceil k \rceil,\alpha_k}(N)\right)^{\alpha_k^2} \leq \mathcal{M}_k(N) \leq \left(\mathcal{M}_{\lfloor k \rfloor,\beta_k}(N)\right)^{\beta_k^2}\]
where $\alpha_k = \sqrt{k/\lceil k \rceil}$ and $\beta_k = \sqrt{k/ \lfloor k \rfloor}$. 

In practice it is often useful to have smoother weights in the Dirichlet polynomial, especially if one is concerned with uniform asymptotics. We will therefore consider 
\begin{equation}\label{eq:twmoment} 
	\mathcal{S}_{k,\varrho}(N) = \lim_{T\to\infty} \frac{1}{T}\int_0^T \left|\sum_{n=1}^N \frac{\varrho^{\Omega(n)}}{n^{1/2+it}}\left(1-\frac{\log{n}}{\log{N}}\right)\right|^{2k}\,dt. 
\end{equation}
In Section~\ref{sec:asymptotic} we will use that $\mathcal{M}_{k,\varrho}(N)\geq \mathcal{S}_{k,\varrho}(N)$ to deduce the desired lower bound in Theorem~\ref{thm:asymptotics} from the following result.
\begin{theorem}\label{thm:twistedmoment} 
	Fix $\delta>0$. Uniformly for $0\leq\varrho \leq \sqrt{2-\delta}$ and positive integers $k = o(\sqrt{\log\log{N}})$ we have the asymptotic 
	\begin{align}
		\mathcal{S}_{k,\varrho}(N) &= a(k,\varrho)\gamma(k,\varrho)(\log{N})^{k^2 \varrho^2} + \mathcal{E}_{k,\varrho}(N) \nonumber \intertext{where the error term satisfies} \mathcal{E}_{k,\varrho}(N)&\ll(\log N)^{k^2\varrho^2-1/2}\exp\left(-k^2\varrho^2\log k-k^2\varrho^2\log\log k+O_\delta(k^2)\right). \nonumber \intertext{The constants are given by} a(k,\varrho) &= \prod_{p} \left(1-\frac{1}{p}\right)^{k^2 \varrho^2} \sum_{j=0}^\infty \frac{d_k^2(p^j) \varrho^{2j}}{p^j}, \nonumber \\
		\gamma(k,\varrho) &= \frac{1}{\Gamma(1+\varrho^2)^{k^2}}\int_{\mathcal{P}_{k,\varrho}}\prod_{i=1}^k \left(1-\sum_{j=1}^k x_{ij}^{1/\varrho^2}\right)\prod_{j=1}^k \left(1-\sum_{i=1}^k x_{ij}^{1/\varrho^2}\right)\,d\underline{x}, \nonumber \intertext{where $\mathcal{P}_{k,\varrho}$ denotes the twisted polytope} \mathcal{P}_{k,\varrho} &= \left\{(x_{ij}) \in \mathbb{R}^{k^2} \,:\, x_{ij}\geq0, \quad \sum_{i=1}^k x_{ij}^{1/\varrho^2} \leq 1, \quad \sum_{j=1}^k x_{ij}^{1/\varrho^2} \leq 1\right\}. 
	\label{eq:twistedpoly} \end{align}
\end{theorem}
\begin{remark}
	If we do not pursue uniform estimates and seek to investigate $\mathcal{M}_{k,\varrho}(N)$ directly, we mention without proof that for a fixed integer $k$ and fixed $0<\varrho<\sqrt{2}$ it is possible to deduce with our techniques that
	\[\lim_{N\to\infty} \frac{\mathcal{M}_{k,\varrho}(N)}{(\log{N})^{k^2\varrho^2}} = a(k,\varrho) \frac{\operatorname{Vol}(\mathcal{P}_{k,\varrho})}{\Gamma(1+\varrho^2)}.\]
	This is an extension of the main result of \cite{CG06}, which corresponds to the case $\varrho=1$, that might be of independent interest. Comparing the twisted polytope $\mathcal{P}_{k,\varrho}$ from \eqref{eq:twistedpoly} to the polytope $\mathcal{P}_k$ from \eqref{eq:Pk} we note the striking geometric effect of the parameter $\varrho$ on the faces of the polytope. 
\end{remark}

From the proof we will see that, in fact, the statement of Theorem~\ref{thm:twistedmoment} holds uniformly for $k^2=o(\sqrt{\log N})$ (and that this can almost certainly be improved, as can the factor of $1/\sqrt{\log N}$ in the error term). However, we have chosen to state it this way since, as we will see later, the main term is of size 
\[(\log N)^{k^2\varrho^2}\exp\left(-k^2\varrho^2\log k-k^2\varrho^2\log\log k+O_\delta(k^2)\right)\]
and so the result would fail to be an asymptotic if $k\geq C\sqrt{\log \log N}$ since the factor of $1/\sqrt{\log N}$ in the error term would be absorbed into $\exp(O_\delta(k^2))$. 

Let us next discuss what happens when $k\to\infty$ and $N$ is fixed. In analogy with \eqref{eq:lindelimit}, we therefore define
\[\mathcal{M}(N) = \lim_{k\to\infty} [\mathcal{M}_k(N)]^{1/(2k)}.\]
A result regarding norms of Dirichlet polynomials (see \cite[Sec.~2.3]{Bayart02}), which is a consequence of their almost periodicity, gives that this limit is equal to 
\begin{equation}\label{eq:truelimit} 
	\mathcal{M}(N) = \sup_{t\geq0} \left|S_N \zeta(1/2+it)\right| = \sum_{n=1}^N \frac{1}{\sqrt{n}} \sim 2\sqrt{N}. 
\end{equation}
Following \cite{FGH07}, we could insert the largest premitted value in the upper bound of Theorem~\ref{thm:asymptotics}, namely $k = C_1\log{N}/\log\log{N}$, to get the upper bound
\[[\mathcal{M}_k(N)]^{1/(2k)} \leq \exp\left(\left(C+o(1)\right)\frac{\log{N}}{\log\log{N}}\right)\]
for some positive constant $C$. However, this is too small compared to the true limit \eqref{eq:truelimit}. This means that the approach to the Lindel\"of conjecture through the Keating--Snaith conjecture discussed in \cite{FGH07} does not carry over to the pseudomoment setting. We are lead to consider the following.
\begin{problem}
	Determine the smallest $k=k(N)$ such that $[\mathcal{M}_k(N)]^{1/2k} \gg \sqrt{N}$. 
\end{problem}

This problem is the final topic of the present paper. By the discussion above, $k = C_1\log{N}/\log\log{N}$ is certainly too small. We will demonstrate that for a general Dirichlet polynomial $f(s) = \sum_{n=1}^N a_n n^{-s}$, the optimal $k$ is $\pi(N)$. For the partial sums $S_N\zeta(1/2+s)$, we can do much better, but we have been unable to resolve the problem. Specifically, we will show that $k= N^\varepsilon$ is sufficient for every $\varepsilon>0$. 

For these arguments, we need to estimate expressions such as 
\begin{equation}\label{eq:limitmeasure} 
	\lim_{T\to\infty} \frac{1}{T}\operatorname{meas}\left(\left\{t \in [0,T] \,:\, \left|\sum_{n=1}^N n^{-1/2-it}\right| \geq \lambda \right\}\right) 
\end{equation}
for large $\lambda$. Our approach is to use an old insight of H. Bohr to translate \eqref{eq:limitmeasure} to the polytorus $\mathbb{T}^d$, for $d=\pi(N)$. Here we will apply a version of Bernstein's inequality for trigonometric polynomials in several variables \cite{QQ13}, Khintchine's inequality \cite{KK01} and estimates for smooth numbers \cite{Granville02}.

\subsection*{Organization} The present paper contains four additional sections. The next section contains some preliminary estimates needed for the proof of Theorem~\ref{thm:twistedmoment}. This proof can be found in Section~\ref{sec:twisted}. Section~\ref{sec:asymptotic} is devoted to the proof of Theorem~\ref{thm:asymptotics}. Finally, in Section~\ref{sec:polytorus} some results on norm comparison for Dirichlet polynomials are obtained.

% ESTIMATES
\section{Preliminary estimates} \label{sec:est}
Our starting point is to expand the square and integrate in the right hand side of \eqref{eq:tmoment} to obtain 
\begin{equation}\label{eq:pseudosum} 
	\mathcal{M}_{k,\varrho}(N)=\sum_{\substack{n_1\cdots n_k=\\n_{k+1}\cdots n_{2k}\\n_j\leq N}}\frac{\varrho^{\Omega(n_1)}\cdots \varrho^{\Omega(n_{2k})}}{(n_1\cdots n_{2k})^{1/2}}. 
\end{equation}
Consider the associated multiple Dirichlet series 
\begin{equation}\label{eq:mds} 
	F_{k,\varrho}(\underline{s})=F_{k,\varrho}(s_1,\ldots,s_{2k})=\sum_{\substack{n_1\cdots n_k= \\n_{k+1}\cdots n_{2k}\\
	n_j\geq 1}} \frac{\varrho^{\Omega(n_1)}\cdots \varrho^{\Omega(n_{2k})}} {n_1^{1/2+s_1}\cdots n_{2k}^{1/2+s_{2k}}}. 
\end{equation}
In preparation for the proof of Theorem~\ref{thm:twistedmoment} in the next section, we will compile some preliminary results and estimates for the Dirichlet series $F_{k,\varrho}$ from \eqref{eq:mds}. Our first lemma relies on a result from \cite{Norton92}. 

\begin{lemma}\label{lem:norton} 
	Fix $\delta>0$. Uniformly for $0\leq \varrho \leq \sqrt{2-\delta}$ and $0<\sigma \leq 1/\log{k}$ it holds that
	\[|F_{k,\varrho}(\underline{s})| \leq \exp\left(-k^2 \varrho^2\left(\log{2\sigma}+\log\log{k^2\varrho^2}+O_\delta(1)\right)\right),\]
	if $s_\ell = \sigma + i t_\ell$ for $1 \leq \ell \leq 2k$.
\end{lemma}
\begin{proof}
	Since $F_{k,\varrho}(\underline{s})$ has positive coefficients, the maximum is attained for $t_\ell=0$. Using that $\sigma_\ell = \sigma$, we find that 
	\begin{multline*}
		|F_{k,\varrho}(\underline{s})| \leq \sum_{\substack{n_1\cdots n_k=\\n_{k+1}\cdots n_{2k}\\
		n_j\geq 1}} \frac{\varrho^{\Omega(n_1)}\cdots \varrho^{\Omega(n_{2k})}}{n_1^{1/2+\sigma}\cdots n_{2k}^{1/2+\sigma}} \\
		= \sum_{n=1}^\infty \frac{d_k^2(n)\varrho^{2\Omega(n)}}{n^{1+2\sigma}} = \prod_p \left(\sum_{j=0}^\infty \frac{d_k^2(p^j) \varrho^{2j}}{p^{(1+2\sigma)j}}\right) 
	\end{multline*}
	We will split the Euler product at $k^2 \varrho^2$. For the small primes, we first use that $\sigma\geq0$ and estimate roughly to find that 
	\begin{align*}
		\prod_{p\leq k^2 \varrho^2} \left(\sum_{j=0}^\infty \frac{d_k^2(p^j) \varrho^{2j}}{p^{j}}\right) &\leq \prod_{p\leq k^2 \varrho^2} \left(\sum_{j=0}^\infty \frac{d_k(p^j) \varrho^j }{p^{j/2}}\right)^2 =\prod_{p\leq k^2 \varrho^2} \left(1-\frac{\varrho}{\sqrt{p}}\right)^{-2k} \\
		&\leq\exp\left(2k C_\delta \sum_{p\leq k^2 \varrho^2} \frac{\varrho}{\sqrt{p}}\right) = \exp\left(O_\delta\left(\frac{k^2 \varrho^2}{\log{k}}\right)\right), 
	\end{align*}
	by the prime number theorem. For the large primes, we use $d_k^2(p^j) \leq d_{k^2}(p^j)$ and the estimate $-\ln(1-x) \leq x + O(x^2)$, for, say $0\leq x \leq 2/3$, to achieve 
	\begin{multline*}
		\prod_{p> k^2 \varrho^2} \left(\sum_{j=0}^\infty \frac{d_k^2(p^j) \varrho^{2j}}{p^{(1+2\sigma)j}}\right) \leq \prod_{p>k^2\varrho^2} \left(1-\frac{\varrho^2}{p^{(1+2\sigma)}}\right)^{-k^2} \\
		=\exp\left(k^2 \varrho^2 \sum_{p>k^2 \varrho^2} \frac{1}{p^{1+2\sigma}} + O(1)\right). 
	\end{multline*}
	We now put into play the following estimate (see~\cite[Lem.~3.12]{Norton92}). Uniformly for $\sigma>0$ and $y\geq2$, it holds that
	\[\sum_{p>y} \frac{1}{p^{1+2\sigma}} = -\log{2\sigma}-\log\log{y} - \gamma + O(\sigma\log y) + O(\sigma^2\log^2 y) + O\left(\frac{1}{\log{y}}\right).\]
	We apply this estimate with $y = k^2 \varrho^2$ and since $\sigma \leq 1/\log{k}$ we get that
	\[k^2 \varrho^2 \sum_{p>k^2 \varrho^2} \frac{1}{p^{1+2\sigma}} = -k^2 \varrho^2\left(\log{2\sigma}+ \log\log{k^2\varrho^2} + O(1)\right),\]
	which completes the proof. 
\end{proof}

Lemma~\ref{lem:norton} will also be used in the proof of the upper bound in Theorem~\ref{thm:asymptotics} found in Section~\ref{sec:asymptotic}. Let us now factor out zeta functions from $F_{k,\varrho}$ and estimate the arithmetic factor $a(k,\varrho)$.

\begin{lemma}\label{lem:mdsfactor} 
	Let $F_{k,\varrho}$ be as in \eqref{eq:mds} for some $0\leq\varrho<\sqrt{2}$ and suppose that $\mre(s_i+s_{j+k})>0$ for $1\leq i,j\leq k$. Then
	\[F_{k,\varrho}(\underline{s})=A_{k,\varrho}(\underline{s})\prod_{i,j=1}^k\zeta(1+s_i+s_{j+k})^{\varrho^2}\]
	where
	\[A_{k,\varrho}(\underline{s}) =\prod_p\prod_{i,j=1}^k\bigg(1-\frac{1}{p^{1+s_i+s_{j+k}}}\bigg)^{\varrho^2} \sum_{\substack{m_1+\cdots+ m_k=\\m_{k+1}+\cdots+ m_{2k}\\
	m_j\geq 0}} \frac{\varrho^{m_1+\cdots+m_{2k}}}{p^{m_1(\frac{1}{2}+s_1)+\cdots +m_{2k}(\frac{1}{2}+s_{2k})}}.\]
	The product is absolutely convergent if $\mre(s_i+s_{j+k})>-1/2$ for $1 \leq i,j \leq k$, and in particular
	\begin{equation} \label{eq:Akrho}
		A_{k,\varrho}(\underline{0}) = \prod_{p} \left(1-\frac{1}{p}\right)^{k^2 \varrho^2} \sum_{j=0}^\infty \frac{d_k^2(p^j) \varrho^{2j}}{p^j} = a(k,\varrho).
	\end{equation}
	Fix $\delta>0$. Uniformly for $0\leq \varrho \leq \sqrt{2-\delta}$ we have that
	\[a(k,\varrho)=\exp\left(-k^2\varrho^2\log(2e^\gamma\log{k\varrho})+O_\delta\left(\frac{k^2}{\log{k}}\right)\right)\]
	as $k\to\infty$.
\end{lemma}
\begin{proof}
	The first statement about the factorization into Euler products is standard and we omit the details of the proof (see e.g.~\cite{CG06,HL16}). 
	
	For the second statement about the asymptotics of $a(k,\varrho)$, we split the Euler product \eqref{eq:Akrho} into two parts as in the proof of Lemma~\ref{lem:norton}. We first consider $p \leq k^2 \varrho^2$ and apply Mertens' third theorem to the effect that 
	\begin{align*}
		\prod_{p\leq k^2 \varrho^2 } \left(1-\frac{1}{p}\right)^{k^2 \varrho^2} &= \left(\frac{e^{-\gamma}}{\log(k^2 \varrho^2)}\right)^{k^2 \varrho^2 }\left(1 + O \left(\frac{1}{\log(k^2 \varrho^2 )}\right)\right)^{k^2\varrho^2} \\
		&= \exp\left(-k^2 \varrho^2 \log\left(2e^{\gamma}\log(k\varrho)\right)+ O\left(\frac{k^2}{\log{k}}\right)\right). 
	\end{align*}
	For the other factor, we recall from the proof of Lemma~\ref{lem:norton} that
	\[\prod_{p\leq k^2 \varrho^2}\left(\sum_{j=0}^\infty \frac{d^2_k(p^j) \varrho^{2j}}{p^j}\right) \leq \exp\left(O_\delta\left(\frac{k^2}{\log{k}}\right)\right).\]
	When $p>k^2 \varrho^2$, we again use that $d_k^2(p^j) \leq d_{k^2}(p^j)$ to obtain the estimate 
	\begin{multline*}
		\log \prod_{p> k^2 \varrho^2} \left(1-\frac{1}{p}\right)^{k^2\varrho^2} \sum_{j=0}^\infty \frac{d_k^2(p^j) \varrho^{2j}}{p^j} \\
		\leq k^2 \sum_{p>k^2 \varrho^2}\left(\varrho^2\log\left(1-\frac{1}{p}\right)- \log\left(1-\frac{\varrho^2}{p}\right)\right) \\
		\ll k^2\sum_{p>k^2 \varrho^2} \frac{1}{p^2} \ll 1. 
	\end{multline*}
	The proof is completed by combining the three estimates. 
\end{proof}

Lemma~\ref{lem:mdsfactor} allows us to extract the behaviour of $F_{k,\varrho}(\underline{s})$ near $\underline{s}=0$ by estimating the Euler product $A_{k,\varrho}$ and the double product of zeta functions separately. We begin with the latter, which is straightforward.

\begin{lemma}\label{lem:zetaprod}
	Suppose that $\mre(s_i+s_{j+k})>0$ for $1\leq i,j\leq k$ and set
	\[S = \max_{1 \leq \ell \leq 2k} |s_\ell|.\]
	If $Sk^2 = o(1)$ and $\varrho\geq0$, then 
	\[\prod_{i,j=1}^k\zeta(1+s_i+s_{j+k})^{\varrho^2} = \left(1 + O_{\varrho}\left(S k^2  \right)\right)\prod_{i,j=1}^{k} \frac{1}{(s_i+{s_{j+k}})^{\varrho^2}}.\]
\end{lemma}

\begin{proof}
	For each zeta function in the double product we apply the expansion
	\[\zeta(1+s)^{\varrho^2}=\frac{1}{s^{\varrho^2}}\left(1+O_{\varrho}(|s|)\right).\]
	Since $s = s_{i}+s_{j+k} \ll S$, we get that
	\[\prod_{i,j=1}^k\zeta(1+s_i+s_{j+k})^{\varrho^2} = \left(1 + O_{\varrho}(S)\right)^{k^2}\prod_{i,j=1}^{k} \frac{1}{(s_i+{s_{j+k}})^{\varrho^2}}\]
	and by the assumption $Sk^2 = o(1)$ we complete the proof.
\end{proof}

The next lemma is the most technical part in the proof of Theorem~\ref{thm:twistedmoment}, and also the part of the argument which forces the restriction $k = o(\sqrt{\log\log{N}})$.
\begin{lemma} \label{lem:Akrhoest}
	Suppose that $\mre(s_i+s_{j+k})\geq0$ for $1 \leq i,j \leq k$ and set
	\[S = \max_{1 \leq \ell \leq 2k} |s_\ell|.\]
	Fix $\delta>0$. If $S \leq 1/\log{k}$ and $0\leq \varrho \leq \sqrt{2-\delta}$ there is a constant $\Sigma_\delta$ so that 
	\[A_{k,\varrho}(\underline{s}) = a(k,\varrho)\left(1+ O\left(S  e^{\Sigma_\delta k^2}\right)\right).\]
\end{lemma}

\begin{proof}
	By the chain rule we get that
	\begin{equation} \label{taylor}
		\begin{split}
			A_{k,\varrho}(\underline{s}) &= A_{k,\varrho}(0) + \int_0^1 \frac{d}{dx} A_{k,\varrho}(x \underline{s})\,dx \\
			&\qquad\qquad\qquad\qquad\qquad= a(k,\varrho) + \sum_{\ell=1}^{2k}\int_0^1  s_\ell \frac{\partial}{\partial s_\ell} A_{k,\varrho}(x \underline{s})\,dx,
		\end{split}
	\end{equation}
	so it remains to show that the partial derivatives satisfy
	\[\left|\frac{\partial}{\partial s_\ell} A_{k,\varrho}(\underline{s})\right| \leq a(k,\varrho) e^{\Sigma_\delta^\prime k^2}\]
	when $S\leq 1/\log{k}$. Note that the factor of $2k$ obtained when we take absolute values of the right hand side of \eqref{taylor} can be absorbed into the exponential. By symmetry, we consider only the case $\ell=1$. We first note that since $\mre(s_i+s_{j+k})\geq0$, we have that
	\[\frac{\partial}{\partial s_1} \log \prod_{i,j=1}^k \left(1-\frac{1}{p^{-1-s_i-s_{j+k}}}\right)^{\varrho^2} = \varrho^2\sum_{j=1}^k p^{-1-s_1-s_{j+k}} \log{p} + O\left(\frac{k\log{p}}{p^2}\right).\]	
	and that
	\begin{multline*}
		\frac{\partial}{\partial s_1} \log{\sum_{\substack{m_1+\cdots+ m_k=\\m_{k+1}+\cdots+ m_{2k}\\
			m_j\geq 0}} \frac{\varrho^{m_1+\cdots+m_{2k}}}{p^{m_1(\frac{1}{2}+s_1)+\cdots +m_{2k}(\frac{1}{2}+s_{2k})}}} \\ 
			= -\varrho^2 \sum_{j=1}^k p^{-1-s_1-s_{j+k}} \log{p} + O\left(e^{C_\delta k} \frac{\log{p}}{p^2}\right).
	\end{multline*}
	Here we used the same trick used on the small primes in Lemma~\ref{lem:norton} and that $\mre(s_i+s_{j+k})\geq0$ twice. Specifically, we estimated
	\[\sum_{\substack{m_1+\cdots+ m_k=\\m_{k+1}+\cdots+ m_{2k}\\
			m_j\geq M}} \frac{\varrho^{m_1+\cdots+m_{2k}}}{p^{m_1(\frac{1}{2}+s_1)+\cdots +m_{2k}(\frac{1}{2}+s_{2k})}}  \ll \left(\frac{k^2 \varrho^2}{p}\right)^M \left(1-\frac{\varrho}{\sqrt{2}}\right)^{-2k}\]
	for $M=2$ in the numerator and $M=1$ in the denominator.
	
	By logarithmic differentiation we therefore obtain
	\[\frac{\partial}{\partial s_1} A_{k,\varrho}(\underline{s}) \ll |A_{k,\varrho}(\underline{s})| \left(e^{C_\delta k}+k\right)\sum_p \frac{\log{p}}{p^2},\]
	and it is sufficient to show that $|A_{k,\varrho}(\underline{s})| \leq A_{k,\varrho}(\underline{0})e^{\Sigma_\delta k^2}$. We will split the ratio $A_{k,\varrho}(\underline{s})/A_{k,\varrho}(\underline{0})$ into four parts, which will be estimated separately. 
	
	\textbf{I.} For primes $p\leq  2k^2$ we use Taylor expansions to estimate
	\begin{multline*}
		\left(1-\frac{1}{p}\right)^{-k^2 \varrho^2} \prod_{i,j=1}^k \left(1-\frac{1}{p^{1+s_i+s_{j+k}}}\right)^{\varrho^2}
				\\=  \exp\left(\varrho^2 \sum_{i,j=1}^k \left(\frac{1}{p}-\frac{1}{p^{1+s_i+s_{j+k}}}\right) + O\left(\frac{k^2 \varrho^2}{p^2}\right)\right)
				\\\leq  \exp\left(\varrho^2 \sum_{i,j=1}^k \frac{C|S|\log{p}}{p} + O\left(\frac{k^2 \varrho^2}{p^2}\right) \right)
	\end{multline*}
	under the assumption that $S\log{p}$ is bounded. Summing over $p\leq 2k^2 $ and using the prime number theorem yields a total contribution
	\[\prod_{p \leq 2k^2} \left(1-\frac{1}{p}\right)^{-k^2 \varrho^2} \prod_{i,j=1}^k \left(1-\frac{1}{p^{1+s_i+s_{j+k}}}\right)^{\varrho^2} \leq \exp(O(k^2)),\]
	where we used that $S\leq 1/\log{k}$.
	
	\textbf{II.} Since $\mre(s_{i}+s_{j+k})\geq0$ we get that
	\begin{equation} \label{eq:bigsumest}
		\Bigg|\sum_{\substack{m_1+\cdots+ m_k=\\m_{k+1}+\cdots+ m_{2k}\\
					m_j\geq 0}} \frac{\varrho^{m_1+\cdots+m_{2k}}}{p^{m_1(\frac{1}{2}+s_1)+\cdots +m_{2k}(\frac{1}{2}+s_{2k})}}\Bigg| \leq \sum_{j=0}^\infty \frac{d_k^2(p^j) \varrho^{2j}}{p^j}
	\end{equation}
			so the ratio between these two are bounded by $1$. We apply this estimate only for $p\leq 2k^2$.
	
	\textbf{III.} For primes $p>2k^2$ we consider
	\begin{equation} \label{eq:largeprimes}
		\prod_{i,j=1}^k\bigg(1-\frac{1}{p^{1+s_i+s_{j+k}}}\bigg)^{\varrho^2} \sum_{\substack{m_1+\cdots+ m_k=\\m_{k+1}+\cdots+ m_{2k}\\
			m_j\geq 0}} \frac{\varrho^{m_1+\cdots+m_{2k}}}{p^{m_1(\frac{1}{2}+s_1)+\cdots +m_{2k}(\frac{1}{2}+s_{2k})}}.
	\end{equation}
	We will use \eqref{eq:bigsumest} combined with the estimate
	\[\Bigg|\prod_{i,j=1}^k\bigg(1-\frac{1}{p^{1+s_i+s_{j+k}}}\bigg)^{\varrho^2}\Bigg| \leq \left(1+\frac{1}{p}\right)^{k^2\varrho^2} \leq \sum_{j=0}^\infty \left|\binom{k^2 \varrho^2}{j}\right| \frac{1}{p^j} \]
	obtained by the fact that $\mre(s_{i}+s_{j+k})\geq0$. Before we apply these estimates, we observe that the first order terms in \eqref{eq:largeprimes} cancel in a similar way to what we found in the logarithmic differentiation above. After combining this observation with the two estimates, we find that the absolute value of \eqref{eq:largeprimes} is smaller than
	\[1 + \sum_{j=2}^\infty \frac{1}{p^j} \sum_{j_1+j_2 = j} \left|\binom{k^2 \varrho^2}{j_1}\right| \varrho^{2j_2} d_k^2(p^{j_2}) \leq 1 + \sum_{j=2}^\infty \frac{(k\varrho)^{2j}(j+1)}{p^j} \leq 1+ \frac{C_\delta k^4 \varrho^4 }{p^2},\]
	where we use that $d_k(p^j) \leq k^j$ and that if $\alpha\geq1$, then
	\[\left|\binom{\alpha}{j}\right| \leq \alpha^j.\]
	We then get a total contribution which is smaller than
	\[\prod_{p>2k^2}\left(1+\frac{C_\delta k^4 \varrho^4}{p^2}\right) \leq \exp\left(C_\delta' k^4 \varrho^4 \sum_{p>2k^2} \frac{1}{p^2}\right) = \exp\left(O_\delta\left(\frac{k^2}{\log{k}}\right)\right).\]
	\textbf{IV.} As the final part in the proof of Lemma~\ref{lem:mdsfactor}, we find that
	\[\prod_{p> 2k^2} \left(1-\frac{1}{p}\right)^{k^2\varrho^2} \sum_{j=0}^\infty \frac{d_k^2(p^j) \varrho^{2j}}{p^j} \geq \prod_{p>2k^2} \left(1-\frac{1}{p}\right)^{k^2 \varrho^2} \left(1+\frac{k^2 \varrho^2}{p}\right) \gg 1\]
	Combining estimates $\textbf{I--IV}$ completes the proof.
\end{proof}

% PROOF OF THEOREM 2
\section{Proof of Theorem~\ref{thm:twistedmoment}} \label{sec:twisted}
Throughout this section, we let $\mathcal{L} = \log{N}$ to simplify various expressions and computations that will appear. We will also assume that $\delta>0$ is fixed and that $0 \leq \varrho \leq \sqrt{2-\delta}$.

To prove Theorem~\ref{thm:twistedmoment}, we first want to express the smoothed version of \eqref{eq:pseudosum} as a $2k$ fold contour integral of \eqref{eq:mds} by applying Perron's formula in each variable $n_\ell$. The smoothing factor yields additional convergence in the integrals that allows us to obtain uniform estimates.

To extract the leading term of this integral our plan is to apply, what is in essence, the saddle point method. This involves identifying the point where the main contribution of the integral arises from, then truncating the integrals at a low height around this point and expanding the integrand in terms of Taylor and Laurent series. After extracting the main term and the arithmetic factor $a(k,\varrho)$, we re-extend the integrals and apply Perron's formula again to compute the geometric factor $\gamma(k,\varrho)$.

To obtain a representation of $\mathcal{S}_{k,\varrho}(N)$ as a $2k$ fold integral, we want to use the following version of Perron's formula. For $c>0$, it holds that
\begin{equation}\label{eq:perrons} 
	\frac{1}{2\pi i} \int_{c-i\infty}^{c+i\infty} x^s \, \frac{ds}{s^2} = 
	\begin{cases}
		\log{x}, & 1\leq x<\infty, \\
		0, & 0<x<1. 
	\end{cases}
\end{equation}
Expanding the integral \eqref{eq:twmoment} as in \eqref{eq:pseudosum}, we find that 
\begin{multline*}
	\mathcal{S}_{k,\varrho}(N) = \sum_{\substack{n_1\cdots n_k=\\n_{k+1}\cdots n_{2k}\\n_j\leq N}}\frac{\varrho^{\Omega(n_1)}\cdots \varrho^{\Omega(n_{2k})}}{(n_1\cdots n_{2k})^{1/2}} \prod_{\ell=1}^{2k} \left(1-\frac{\log{n_\ell}}{\log{N}}\right)	
	%\\ = \frac{1}{(\log{N})^{2k}}\sum_{\substack{n_1\cdots n_k=\\n_{k+1}\cdots n_{2k}\\n_j\leq N}}\frac{\varrho^{\Omega(n_1)}\cdots \varrho^{\Omega(n_{2k})}}{(n_1\cdots n_{2k})^{1/2}} \prod_{\ell=1}^{2k} \log\left(\frac{N}{n_\ell}\right) 
	\\	= \frac{1}{(\log{N})^{2k} (2\pi i )^{2k}} \int_{c-i\infty}^{c+i\infty} \cdots \int_{c-i\infty}^{c+i\infty} F_{k,\varrho}(\underline{s}) \prod_{\ell=1}^{2k} N^{s_\ell} \frac{ds_\ell}{s_\ell^2}, 
\end{multline*}
where we applied \eqref{eq:perrons} with $x= N/n_\ell$ for $1\leq \ell \leq 2k$. We now substitute $s_\ell \mapsto s_\ell/\mathcal{L}$ and find that
\begin{equation} \label{eq:Skrhoint}
	\mathcal{S}_{k,\varrho}(N) = \frac{1}{(2\pi i)^{2k}} \int_{\sigma-i\infty}^{\sigma+i\infty} \cdots \int_{\sigma-i\infty}^{\sigma+i\infty} F_{k,\varrho}(\underline{s}/\mathcal{L}) \prod_{\ell=1}^{2k} e^{s_\ell}\, \frac{ds_\ell}{s_\ell^2}
\end{equation}
where $\sigma = \mathcal{L} c$. Our goal is now to choose $\sigma>0$ and truncate the integrals at a suitable height $\mathcal{T}$. Specifically, we will obtain the following result.
\begin{lemma}\label{lem:truncit} 
	If $k = o(\mathcal{T})$, then
	\[\mathcal{S}_{k,\varrho}(N) = \frac{1}{(2\pi i)^{2k}} \int_{k-i\mathcal{T}}^{k+i\mathcal{T}} \cdots\int_{k-i\mathcal{T}}^{k+i\mathcal{T}} F_{k,\varrho}(\underline{s}/\mathcal{L}) \prod_{\ell=1}^{2k} e^{s_\ell}\, \frac{ds_\ell}{s_\ell^2} 
		+ O\left(\frac{\mathcal{C}_{k,\varrho}(N)}{k^{2k-1}\mathcal{T}}\right)\]
	where $\mathcal{C}_{k,\varrho}(N) = |e^{2k^2} F_{k,\varrho}\left(k/\mathcal{L},\ldots,k/\mathcal{L}\right)|$. If $k \leq C_1 \log{N}/\log\log{N}$, then 
	\begin{equation} \label{eq:Ckrho}
		\mathcal{C}_{k,\varrho}(N) \leq (\log{N})^{k^2\varrho^2}\exp\left(-k^2 \varrho^2\left(\log{k}+\log\log{k}+O_\delta(1)\right)\right).
	\end{equation}
\end{lemma}

\begin{proof}
	Let us first explain the choice $\sigma=k$. From Lemma~\ref{lem:norton} we get that
	\[\left|F_{k,\varrho}(\underline{s}/\mathcal{L}) \prod_{\ell=1}^{2k} e^{s_\ell}\right| \leq \exp\left(2k \sigma -k^2\varrho^2 \left(\log(2\sigma/\mathcal{L})+\log\log{k^2\varrho^2} + O_\delta(1)\right)\right),\]
	provided $\sigma/\mathcal{L} \leq 1/\log{k}$. A calculus argument gives that the optimal value is $2\sigma = k \varrho^2$, but we will for notational simplicity use $\sigma=k$. The effect of this suboptimal choice is absorbed in the $O_\delta(1)$ term. If $k \leq C_1\log{N}/\log\log{N}$ then $\sigma/\mathcal{L}\leq 1/\log{k}$, and we obtain \eqref{eq:Ckrho}.
	
	We now consider the error when truncating \eqref{eq:Skrhoint} at height $t_\ell = \mathcal{T}$ for each integral. We take absolute values inside the integrals and extract $\mathcal{C}_{k,\varrho}(N)$. What remains are $2^{2k-1}$ combinations of integrals of the following types:
	\begin{align}
		\frac{1}{2\pi} \int_{k-i\mathcal{T}}^{k+i\mathcal{T}} \frac{ds}{|s|^2} &\leq \frac{1}{2\pi}\int_{k-i\infty}^{k+i\infty} \frac{ds}{|s|^2} = \frac{1}{2k} \label{eq:firstint} \\
		\frac{1}{2\pi} \left(\int_{k-i\infty}^{k-i\mathcal{T}} + \int_{k+ i\mathcal{T}}^{k+ i\infty}\right) \frac{ds}{|s|^2} &= \frac{1}{\pi k}\left(\frac{\pi}{2}-\arctan(\mathcal{T}/k)\right) \leq \frac{C}{\mathcal{T}} \label{eq:secondint}
	\end{align}
	since $k = o(\mathcal{T})$ and
	\[\arctan{x} = \frac{\pi}{2}-\frac{1}{x}+O(x^{-3})\]
	as $x\to\infty$. Since $k = o(\mathcal{T})$, the integrals \eqref{eq:secondint} are smaller than the integrals \eqref{eq:firstint}. Hence the largest contribution from the error is obtained by choosing the maximal number of integrals like \eqref{eq:firstint}. However, there is always at least one integral like \eqref{eq:secondint}, so we conclude that the total error is at most
	\[\mathcal{C}_{k,\varrho}(N) \times 2^{2k-1} \times \frac{1}{(2k)^{2k-1}} \frac{C}{\mathcal{T}} \ll \frac{\mathcal{C}_{k,\varrho}(N)}{k^{2k-1} \mathcal{T}}\]
	as desired.	
\end{proof}

We will now investigate the integral
\begin{equation} \label{eq:IkrhoNT}
	\mathcal{I}_{k,\varrho}(N,\mathcal{T}) = \frac{1}{(2\pi i)^{2k}} \int_{k-i\mathcal{T}}^{k+i\mathcal{T}} \cdots\int_{k-i\mathcal{T}}^{k+i\mathcal{T}} F_{k,\varrho}(\underline{s}/\mathcal{L}) \prod_{\ell=1}^{2k} e^{s_\ell}\, \frac{ds_\ell}{s_\ell^2},
\end{equation}
appearing in Lemma~\ref{lem:truncit}, where the parameter $\mathcal{T}$ will be chosen later. To extract the main term from this integral, we will apply Lemma~\ref{lem:mdsfactor}, Lemma~\ref{lem:zetaprod} and Lemma~\ref{lem:Akrhoest}.

\begin{lemma} \label{lem:JkrhoT}
	Let $\mathcal{I}_{k,\varrho}(N,\mathcal{T})$ be as in \eqref{eq:IkrhoNT} and set
	\begin{equation} \label{eq:JkrhoT}
		\mathcal{J}_{k,\varrho}(\mathcal{T}) = \frac{1}{(2\pi i)^{2k}}\int_{k-i\mathcal{T}}^{k+i\mathcal{T}}\cdots \int_{k-i\mathcal{T}}^{k+i\mathcal{T}} \prod_{i,j=1}^k\frac{1}{(s_i+s_{j+k})^{\varrho^2}} \prod_{\ell=1}^{2k} e^{s_\ell}\frac{ds_\ell}{s_\ell^2}.
	\end{equation}
	Suppose that $k = o(\mathcal{T})$, and that $k^2\mathcal{T}/\mathcal{L} = o(1)$. Then
	\[\mathcal{I}_{k,\varrho}(N,\mathcal{T}) = \mathcal{L}^{k^2\varrho^2} a(k,\varrho) \left(\mathcal{J}_{k,\varrho}(\mathcal{T})+ O\left(\frac{\mathcal{T}}{\mathcal{L}} \frac{e^{ k^2(\Sigma_\delta +2 + o(1))}}{(2k)^{k^2\varrho^2}} \right)\right).\]
\end{lemma}

\begin{proof}
	With the assumptions on $k$ and $\mathcal{T}$ we get from Lemma~\ref{lem:mdsfactor}, Lemma~\ref{lem:zetaprod} and Lemma~\ref{lem:Akrhoest} that in the domain of integration it holds that
	\[F_{k,\varrho}(\underline{s}/\mathcal{L}) = \mathcal{L}^{k^2 \varrho^2} a(k,\varrho) \prod_{i,j=1}^k\frac{1}{(s_i+s_{j+k})^{\varrho^2}} \left(1 + O\left(\frac{\mathcal{T}}{\mathcal{L}} e^{k^2(\Sigma_\delta+o(1))}\right)\right).\]
	We complete the proof by noting that
	\[\left|\frac{1}{(2\pi i)^{2k}}\int_{k-i\mathcal{T}}^{k+i\mathcal{T}}\cdots \int_{k-i\mathcal{T}}^{k+i\mathcal{T}}\prod_{i,j=1}^k\frac{1}{(s_i+s_{j+k})^{\varrho^2}} \prod_{\ell=1}^{2k} e^{s_\ell}\frac{ds_\ell}{s_\ell^2}\right| \leq \frac{e^{2k^2}}{(2k)^{k^2\varrho^2}} \frac{1}{(2k)^{2k}},\]
	where we move the absolute values inside and use \eqref{eq:firstint}.
\end{proof}

We now use Perron's formula in reverse to extract the geometric factor from the integral appearing in Lemma~\ref{lem:JkrhoT}.

\begin{lemma} \label{lem:geo}
	If $k = o(\mathcal{T})$ and $\mathcal{J}_{k,\varrho}(\mathcal{T})$ is as in \eqref{eq:JkrhoT}, then
	\[\mathcal{J}_{k,\varrho}(\mathcal{T}) = \gamma(k,\varrho)+O\left(\frac{1}{\mathcal{T}}\frac{e^{k^2(2+o(1))}}{(2k)^{k^2\varrho^2}}\right)\]
	where $\gamma(k,\varrho)$ is the geometric factor in Theorem~\ref{thm:twistedmoment}.
\end{lemma}
\begin{proof}
	We first re-extend the integrals and estimate as in Lemma~\ref{lem:truncit}. Since $\mre(s_\ell)=k$ for $1\leq \ell \leq 2k$, we have that 
	\[\left|\prod_{i,j=1}^k\frac{1}{(s_i+s_{j+k})^{\varrho^2}} \prod_{\ell=1}^{2k} e^{s_\ell}\right|\leq \frac{e^{2k^2}}{(2k)^{k^2\varrho^2}}.\]
	We then follow the second part of the proof of Lemma~\ref{lem:truncit} line for line to obtain the stated error term. What remains is to show that $\gamma(k,\varrho) = \mathcal{J}_{k,\varrho}$ where
	\begin{equation} \label{eq:toshow}
		\mathcal{J}_{k,\varrho} = \frac{1}{(2\pi i)^{2k}}\int_{k-i\infty}^{k+i\infty}\cdots \int_{k-i\infty}^{k+i\infty} \prod_{i,j=1}^k\frac{1}{(s_i+s_{j+k})^{\varrho^2}} \prod_{\ell=1}^{2k} e^{s_\ell}\frac{ds_\ell}{s_\ell^2}.
	\end{equation}
	We begin with the integral representation of the gamma function,
	\[\Gamma(\varrho^2) = \int_0^\infty e^{-x} x^{\varrho^2}\,\frac{dx}{x}\]
	and substitute $x \mapsto (s_i+s_{j+k}) x_{ij}$ for $\mre(s_i+s_{j+k})>0$ which gives that
	\[\frac{1}{(s_i+s_{j+k})^{\varrho^2}} = \frac{1}{\Gamma(\varrho^2)} \int_0^\infty e^{-(s_i+s_{j+k})x_{ij}}x_{ij}^{\varrho^2}\,\frac{dx_{ij}}{x_{ij}}.\]
	For each term of the $k^2$ factors in the product over $i,j$ in \eqref{eq:toshow}, we apply this identity to the effect that 
	\begin{multline*}
		\mathcal{J}_{k,\varrho}=\frac{1}{\Gamma(\varrho^2)^{k^2}}\int_0^\infty \cdots \int_0^\infty \frac{1}{(2\pi i)^{2k}}\int_{k-i\infty}^{k+i\infty}\cdots \int_{k-i\infty}^{k+i\infty} \\\times\left[\prod_{i=1}^k e^{s_i\left(1-\sum_{j=1}^k x_{ij}\right)} \prod_{j=1}^k e^{s_{j+k}\left(1-\sum_{i=1}^k x_{ij}\right)}\right]\, \prod_{\ell=1}^{2k} \frac{ds_\ell}{s_\ell^2}\,\prod_{i,j=1}^k x_{ij}^{\varrho^2} \frac{dx_{ij}}{x_{ij}}, 
	\end{multline*}
	where the interchange in order of integration is valid by absolute convergence. The $s_\ell$-integrals are now separable and so here we may apply \eqref{eq:perrons} in the form
	\[\frac{1}{2\pi i }\int_{c-i\infty}^{c+i\infty}e^{s(1-X)}\frac{ds}{s^2} =
	\begin{cases}
		1-X, &\text{if } X\leq 1,\\
		0, &\text{if } X>1,
	\end{cases}
	\]
	with $c=k>0$ in each variable to find that 
	\[\mathcal{J}_{k,\varrho}=\frac{1}{\Gamma(\varrho^2)^{k^2}} \int_{\mathcal{P}_k} \prod_{i=1}^k \left(1-\sum_{j=1}^k x_{ij}\right) \prod_{j=1}^k \left(1-{\sum_{i=1}^k} x_{ij}\right)\,\prod_{i,j=1}^k x_{ij}^{\varrho^2} \frac{dx_{ij}}{x_{ij}}\]
	where $\mathcal{P}_k$ is the polytope \eqref{eq:Pk}. We then apply the substitution $x_{ij}^{\varrho^2} \mapsto x_{ij}$ in each variable to find that
	\[\mathcal{J}_{k,\varrho}=\frac{1}{\Gamma(1+\varrho^2)^{k^2}}\int_{\mathcal{P}_{k,\varrho}} \prod_{i=1}^k \left(1-\sum_{j=1}^k x_{ij}^{1/\varrho^2}\right) \prod_{j=1}^k \left(1-{\sum_{i=1}^k} x_{ij}^{1/\varrho^2}\right) d\underline{x}\]
	where $\mathcal{P}_{k,\varrho}$ is the twisted polytope \eqref{eq:twistedpoly} and so $\mathcal{J}_{k,\varrho}=\gamma(k,\varrho)$ as desired.
\end{proof}

By combining all the results of this section, we finally obtain a proof of Theorem~\ref{thm:twistedmoment}.

\begin{proof}[Final part in the proof of Theorem~\ref{thm:twistedmoment}]
	By Lemma~\ref{lem:truncit}, Lemma~\ref{lem:JkrhoT} and Lemma~\ref{lem:geo}, we get the desired main term with an error term that satisfies
	\[\mathcal{E}_{k,\varrho}(N) \ll \frac{\mathcal{C}_{k,\varrho}(N)}{k^{2k-1}\mathcal{T}} + \mathcal{L}^{k^2\varrho^2} a(k,\varrho)\left(\frac{\mathcal{T}}{\mathcal{L}} \frac{e^{ k^2(\Sigma_\delta +2 + o(1))}}{(2k)^{k^2\varrho^2}} +\frac{1}{\mathcal{T}}\frac{e^{k^2(2+o(1))}}{(2k)^{k^2\varrho^2}}\right),\]
	so we choose $\mathcal{T} = \sqrt{\mathcal{L}}$, recall \eqref{eq:Ckrho} and Lemma~\ref{lem:mdsfactor} to obtain
	\[\mathcal{E}_{k,\varrho}(N) \ll \mathcal{L}^{k^2\varrho^2-1/2}\exp\left(-k^2\varrho^2(\log{k}+\log\log{k}+O_\delta(1))\right)\]
	provided $k^2 = o(\sqrt{\mathcal{L}})$. To ensure that this is smaller than the main term, we require that $\mathcal{L}^{-1/2} e^{O_\delta(k^2)} \to 0$, which means that $k = o(\sqrt{\log\log{N}})$. 
\end{proof}

% PROOF OF THEOREM 1
\section{Proof of Theorem~\ref{thm:asymptotics}} \label{sec:asymptotic} 
Let $0<q<\infty$ and define 
\begin{equation}\label{eq:Lq} 
	\|f\|_q = \left(\lim_{T\to\infty} \frac{1}{T}\int_0^T |f(it)|^q \,dt \right)^\frac{1}{q} 
\end{equation}
for $f(s) = \sum_{n=1}^N a_n n^{-s}$. The limit exists for any Dirichlet polynomial $f$ and every $0<q<\infty$ (see \cite{Bayart02}). We also set 
\begin{equation}\label{eq:Linf} 
	\|f\|_\infty = \sup_{t\in\mathbb{R}} |f(it)|, 
\end{equation}
and recall that $\|f\|_q \to \|f\|_\infty$ as $q\to\infty$. Note that the pseudomoments we are interested in \eqref{eq:pseudomoments} can alternatively be expressed as 
\begin{equation}\label{eq:pseudonorm} 
	\mathcal{M}_k(N) = \|f_N\|_{2k}^{2k} 
\end{equation}
for $f_N(s) = S_N \zeta(1/2+s)$ and $\mathcal{M}(N) = \|f_N\|_\infty$ in light of \eqref{eq:truelimit}.

Let $\varrho$ be any non-negative real number, and set
\[\mathcal{W}_\varrho f(s) = \sum_{n=1}^N \varrho^{\Omega(n)} a_n n^{-s}.\]
The following version of Weissler's inequality \cite{Weissler80} for Dirichlet polynomials can be extracted from \cite[Sec.~3]{Bayart02}. 
\begin{lemma}[Weissler's inequality] \label{lem:weissler} 
	Suppose that $0<q_1\leq q_2 < \infty$ and let $0<\varrho \leq \sqrt{q_1/q_2}$. Then
	\[\|\mathcal{W}_\varrho f\|_{q_2} \leq \|f\|_{q_1}\]
	for every Dirichlet polynomial $f(s) = \sum_{n=1}^N a_n n^{-s}$. 
\end{lemma}
Our plan is to use Lemma~\ref{lem:weissler} to relate $\mathcal{M}_k(N)$ for non-integers $k\geq1$ to the twisted moments $\mathcal{M}_{\lfloor k \rfloor,\varrho_1}(N)$ and $\mathcal{M}_{\lceil k \rceil, \varrho_2}(N)$. 

\subsection{Proof of the upper bound in Theorem~\ref{thm:asymptotics}} We begin with the proof of the upper bound in Theorem~\ref{thm:asymptotics}, which will be deduced from Lemma~\ref{lem:norton}, Lemma~\ref{lem:weissler} and Rankin's trick.
\begin{lemma}\label{lem:Mkrhoupper} 
	Fix $\delta>0$ and suppose that $0 < \varrho \leq \sqrt{2-\delta}$. Uniformly for every integer $1 \leq k \leq C_1\log{N}/\log\log{N}$ we have that
	\[\frac{\mathcal{M}_{k,\varrho}(N)}{(\log{N})^{k^2 \varrho^2}} \leq \exp\left(-k^2 \varrho^2 \left(\log(k^2 \varrho) + \log\log(k \varrho) + O_\delta(1)\right)\right).\]
\end{lemma}
\begin{proof}
	Let
	\[d_{k,N}(n)=\sum_{\substack{n_1\cdots n_k=n \\
	n_j\leq N}} 1.\]
	We rewrite \eqref{eq:pseudosum} using $d_{k,N}$ and apply Rankin's trick to the effect that 
	\[\mathcal{M}_{k,\varrho}(N) = \sum_{n=1}^{N^k} \frac{d_{k,N}^2(n) \varrho^{2\Omega(n)}}{n} \leq N^{2k\sigma } \sum_{n=1}^{\infty} \frac{d_k^2(n)\varrho^{2\Omega(n)}}{n^{1+2\sigma}} = N^{2k\sigma } F(\sigma,\ldots, \sigma). \]
	We then apply Lemma~\ref{lem:norton} with $\sigma=k\varrho^2/\log N$. The requirement $\sigma \leq 1/\log{k}$ from Lemma~\ref{lem:norton} is satisfied for $k \leq C_1 \log{N}/\log\log{N}$.
\end{proof}
\begin{proof}
	[Proof of the upper bound in Theorem~\ref{thm:asymptotics}] If $k\geq2$ is an integer, we directly use Lemma~\ref{lem:Mkrhoupper} with $\varrho=1$. If $k\geq2$ is not an integer, we first use \eqref{eq:pseudonorm} and Lemma~\ref{lem:weissler} with $q_1 = 2\lfloor k \rfloor$, $q_2 = 2k$ and $\varrho = \sqrt{k / \lfloor k \rfloor}$ (in reverse) to obtain
	\[\frac{\mathcal{M}_k(N)}{(\log{N})^{k^2}} \leq \frac{\left(\mathcal{M}_{\lfloor k\rfloor ,\varrho}(N)\right)^{k/\lfloor k \rfloor}}{(\log{N})^{k^2}} = \left(\frac{\mathcal{M}_{\lfloor k \rfloor,\varrho}(N)}{(\log{N})^{\lfloor k \rfloor^2 \varrho^2}}\right)^{\varrho^2}.\]
	We then use Lemma~\ref{lem:Mkrhoupper} and that $\lfloor k \rfloor^2 \varrho^4 = k^2$, to obtain
	\[\frac{\mathcal{M}_k(N)}{(\log{N})^{k^2}} \leq \exp\left(-k^2\left(\log{k}+\log\log{k} + C_\delta \right)\right)\]
	uniformly for $\lfloor k \rfloor \leq C_1\log{N}/\log\log{N}$. Since $k\geq2$, we applied Lemma~\ref{lem:weissler} with $\varrho = \sqrt{k/\lfloor k\rfloor} \leq \sqrt{3/2}$ and the requirement $0 \leq \varrho \leq \sqrt{2-\delta}$ of Lemma~\ref{lem:Mkrhoupper} is satisfied with e.g. $\delta=1/2$. 
	
	For the second statement, we check that if $k = C_2\log{N}/\log\log{N}$, then
	\[\log\log{N}-\log{k}-\log\log{k} = -\log{C_2}-\log\left(1+\frac{\log{C_2}-\log\log\log{N}}{\log\log{N}}\right).\]
	The assumption that the upper bound in Theorem~\ref{thm:asymptotics} holds the prescribed value of $k$ yields that 
	\begin{multline*}
		\mathcal{M}_k(N) \leq (\log{N})^{k^2} \exp\left(-k^2\log{k}-k^2\log\log{k}+Ck^2\right) \\
		= \exp\left(k^2\left(-\log{C_2}+C + o(1)\right)\right) 
	\end{multline*}
	which contradicts the trivial bound $\mathcal{M}_k(N)\geq1$ as $N\to\infty$ if $C_2>e^C$. 
\end{proof}

\subsection{Proof of the lower bound in Theorem~\ref{thm:asymptotics}} For the proof of the lower bound in Theorem~\ref{thm:asymptotics}, we require the estimate
\[a(k,\varrho)\gamma(k,\varrho)=\exp(-k^2\varrho^2\log k-k^2\varrho^2\log\log k+O_\delta(k^2))\]
where $a(k,\varrho)$ and $\gamma(k,\varrho)$ are the constants appearing in Theorem~\ref{thm:twistedmoment}. The first factor is handled by Lemma~\ref{lem:mdsfactor} and the second we will deduce below. Once we know these estimates, we will apply Lemma~\ref{lem:weissler} after relating $\mathcal{M}_{k,\varrho}(N)$ to the smoothed moments $\mathcal{S}_{k,\varrho}(N)$. 

A precise asymptotic expansion for the volume of the Birkhoff polytope
\[\mathcal{B}_k =\left\{(x_{ij}) \in \mathbb{R}^{k^2} \,:\, x_{ij}\geq0, \quad \sum_{i=1}^k x_{ij} = 1, \quad \sum_{j=1}^k x_{ij} = 1\right\},\]
which appear in the formulas for the moments \eqref{eq:sigma0} when $k\geq1$ is an integer (see \cite{HNR,HL16}), can be found in \cite{CM09}. Extracting the first two terms of this formula gives
\[\log{\operatorname{Vol}(\mathcal{B}_k)} = -k^2\log{k}+k^2 + O(k\log{k}).\]
We have been unable to find a similar result for the polytope $\mathcal{P}_k$ from \eqref{eq:Pk}, however one can extract a similar result\footnote{One finds that, to leading order, $\log \operatorname{Vol}(\mathcal{P}_k)={-k^2\log{k}+O(k^2)}$.} from the following proof.
\begin{lemma}\label{lem:volumeest} 
	Suppose that $\varrho>0$ and that $k$ is a positive integer. Then
	\[\Gamma(1+\varrho^2)^{-k^2}2^{-2k-k^2 \varrho^2}k^{-k^2 \varrho^2} \leq \gamma(k,\varrho) \leq \left(\Gamma(1+k \varrho^2)\right)^{-k}.\]
	In particular, if $\varrho$ is bounded and $k\to\infty$, we have that 
	\begin{equation}\label{eq:stirlings} 
		\log\left(\gamma(k,\varrho)\right) =- k^2 \varrho^2 \log{k} + O(k^2). 
	\end{equation}
\end{lemma}
\begin{proof}
	Recall that
	\[\gamma(k,\varrho) = \frac{1}{\Gamma(1+\varrho^2)^{k^2}}\int_{\mathcal{P}_{k,\varrho}}\prod_{i=1}^k \left(1-\sum_{j=1}^k x_{ij}^{1/\varrho^2}\right)\prod_{j=1}^k \left(1-\sum_{i=1}^k x_{ij}^{1/\varrho^2}\right)\,d\underline{x}\]
	where $\mathcal{P}_{k,\varrho}$ is the twisted polytope \eqref{eq:twistedpoly}. We will find smaller and larger sets where the integrand can be easily estimated and the volume easily computed. 
	\begin{align*}
		\mathcal{L}_{k,\varrho} &= \left\{(x_{ij}) \in \mathbb{R}^{k^2}\,:\, 0 \leq x_{ij} \leq \frac{1}{(2k)^{\varrho^2}}\right\}, \\
		\mathcal{U}_{k,\varrho} &= \left\{(x_{ij}) \in \mathbb{R}^{k^2}\,:\, x_{ij}\geq0, \quad \sum_{j=1}^k x_{ij}^{1/\varrho^2} \leq 1\right\}. 
	\end{align*}
	Clearly $\mathcal{L}_{k,\varrho} \subset \mathcal{P}_{k,\varrho} \subset \mathcal{U}_{k,\varrho}$, so we obtain the lower and upper bounds for $\gamma(k,\varrho)$ by integrating over $\mathcal{L}_{k,\varrho}$ and $\mathcal{U}_{k,\varrho}$, respectively. Now,
	\[\int_{\mathcal{L}_{k,\varrho}}\prod_{i=1}^k \left(1-\sum_{j=1}^k x_{ij}^{1/\varrho^2}\right)\prod_{j=1}^k \left(1-\sum_{i=1}^k x_{ij}^{1/\varrho^2}\right)\,d\underline{x} \geq \left(1-\frac{1}{2}\right)^{2k} \operatorname{Vol}(\mathcal{L}_{k,\varrho})\]
	and clearly $\operatorname{Vol}(\mathcal{L}_{k,\varrho}) = (2k)^{-k^2\varrho^2}$. Combined, this yields the lower bound
	\[\gamma(k,\varrho) \geq \Gamma(1+\varrho^2)^{-k^2}2^{-2k-k^2 \varrho^2}k^{-k^2 \varrho^2}.\]
	For the upper bound, we simply use that the integrand is bounded from above by $1$ to obtain
	\[\gamma(k,\varrho) \leq \frac{\operatorname{Vol}(\mathcal{U}_{k,\varrho})}{\Gamma(1+\varrho^2)^{k^2}}.\]
	We then use the well-known formula for the volume of the $\ell^r$-ball in $\mathbb{R}^n$,
	\[\operatorname{Vol}\left(\left\{x \in \mathbb{R}^n \,:\, \sum_{j=1}^n |x_j|^r \leq 1\right\}\right) = \frac{2^n \left(\Gamma(1+1/r)\right)^n}{\Gamma(1+n/r)},\]
	with $r = 1/\varrho^2$ and $n=k$ which gives a total upper bound of
	\[\gamma(k,\varrho) \leq \frac{1}{\Gamma(1+k \varrho^2)^k}.\]
	The upper bound in \eqref{eq:stirlings} is deduced by
	\[\frac{1}{(\Gamma(1+k \varrho^2))^k} \leq \exp\left(-k^2 \varrho^2 \log{k} + k^2 \varrho^2(1-2\log{\varrho})\right),\]
	where we used Stirling's formula of the form $\Gamma(1+x) \geq (x/e)^x$. 
\end{proof}
\begin{proof}
	[Proof of the lower bound in Theorem~\ref{thm:asymptotics}] For $k\geq2$, we set $\varrho = \sqrt{k/\lceil k \rceil}$ and combine \eqref{eq:pseudonorm} with Lemma~\ref{lem:weissler} to find that
	\[\mathcal{M}_k(N) \geq \left(\mathcal{M}_{\lceil k \rceil,\varrho}(N)\right)^{\varrho^2} \geq \left(\mathcal{S}_{\lceil k \rceil, \varrho}(N)\right)^{\varrho^2}\]
	since clearly $1-\log{n}/\log{N}\leq1$ when $1\leq n \leq N$. From Theorem~\ref{thm:twistedmoment}, we know uniformly for $k = o(\sqrt{\log\log{N}})$ that
	\[\mathcal{S}_{\lceil k \rceil, \varrho}(N) \sim (\log{N})^{\lceil k \rceil^2 \varrho^2} a(\lceil k \rceil, \varrho) \gamma(\lceil k \rceil, \varrho).\]
	However, by choosing some sufficiently small $c$, we may ensure that absolute value of the error term in Theorem~\ref{thm:twistedmoment} is smaller than, say, $1/2$ times the main term uniformly for $2\leq k \leq c \sqrt{\log\log{N}}$. Using the fact that $\lceil k \rceil^2 \varrho^4 = k^2$ and the estimates from Lemma~\ref{lem:mdsfactor} and Lemma~\ref{lem:volumeest}, we can now complete the proof by similar computations as in the proof of the upper bound presented above. 
\end{proof}

\section{Norm comparisons for Dirichlet polynomials} \label{sec:polytorus} 
For an integer $d\geq1$, let $\mathbb{T}^d$ denote the polytorus
\[\mathbb{T}^d = \left\{z = (z_1,\ldots,z_d)\,:\, |z_j|=1 \text{ for } 1\leq j \leq d\right\},\]
and for $0<q\leq \infty$ let $L^q(\mathbb{T}^d)$ denote the usual $L^q$ space with respect to the normalized product Lebesgue arc measure $\mu_d(z) = \mu(z_1)\times \cdots \times \mu(z_d)$. The main tool of the present section is the Bohr correspondence, which allows us to compute the limit measure \eqref{eq:limitmeasure} and the norms \eqref{eq:Lq} and \eqref{eq:Linf} on the polytorus $\mathbb{T}^d$.

Let $N\geq2$ be a positive integer and set $d=\pi(N)$. Every positive integer $n \leq N$ has the unique prime factorization
\[n = \prod_{j=1}^d p_j^{\kappa_j}.\]
The Dirichlet polynomial $f(s) = \sum_{n=1}^N a_n n^{-s}$ corresponds to a polynomial in $d$ variables by replacing each prime number $p_j^{-s}$ with an independent complex variable $z_j$. Specifically, we set
\[F(z) = \sum_{n=1}^N a_n z(n), \qquad\text{where}\qquad z(n) = \prod_{j=1}^d z_j^{\kappa_j}.\]
By Kronecker's theorem, the flow 
\begin{equation}\label{eq:Kflow} 
	\mathcal{T}_d(t) = \left(2^{-it},\,3^{-it},\,\ldots,\,p_d^{-it}\right) 
\end{equation}
is dense on $\mathbb{T}^d$. This implies that the norm \eqref{eq:Linf} is preserved under the Bohr correspondence,
\[\|f\|_\infty = \sup_{z\in\mathbb{T}^d} |F(z)| = \|F\|_{L^\infty(\mathbb{T}^d)}.\]
Using the ergodic theorem for the Kronecker flow \eqref{eq:Kflow}, it is proved in \cite{Bayart02} that the norms \eqref{eq:Lq} are also preserved, 
\begin{equation}\label{eq:polytorus} 
	\|f\|_q = \|F\|_{L^q(\mathbb{T}^d)} = \left(\int_{\mathbb{T}^d} |F(z)|^q\,d\mu_d(z)\right)^{\frac{1}{q}}. 
\end{equation}
The same ergodic argument also gives that
\[\lim_{T\to\infty} \frac{1}{T}\operatorname{meas}\left(\left\{t \in [0,T] \,:\, |f(it)| \geq \lambda \right\}\right) = \mu_d \left(\left\{z \in \mathbb{T}^d \,:\, |F(z)|\geq \lambda\right\}\right).\]
A more elementary proof of \eqref{eq:polytorus} can be found in \cite[Sec.~3]{SS09}. 

We are mainly interested in comparing $\|f_N\|_{2k}$ and $\|f_N\|_\infty$ for the specific Dirichlet polynomial $f_N(s)=S_N \zeta(1/2+s)$. However, we first investigate the case of a general Dirichlet polynomial $f(s) = \sum_{n=1}^N a_n n^{-s}$. We will apply a version of Bernstein's inequality for trigonometric polynomials in several variables (see~\cite[Sec.~5.2]{QQ13}). 
\begin{lemma}\label{lem:bernstein} 
	Let $F(z)$ be a polynomial of degree $k$ in $d$ variables. Then
	\[\left|F(e^{i\theta_1},\ldots,e^{i\theta_d}) - F(e^{i\vartheta_1},\ldots,e^{i\vartheta_d})\right| \leq \frac{\pi}{2} k \|F\|_{L^\infty(\mathbb{T}^d)} \sup_{1\leq j \leq d} |\theta_j - \vartheta_j|,\]
	for every $\theta=(\theta_1,\ldots,\theta_d)$ and $\vartheta = (\vartheta_1,\ldots,\vartheta_d)$ in $(\mathbb{R}/[0,2\pi))^{d}$. 
\end{lemma}
\begin{remark}
	It is not known whether the constant $\pi/2$ in Lemma~\ref{lem:bernstein} can be replaced with the smaller constant $1$, which is the correct statement for $d=1$. If this indeed holds, then $\pi^2$ in Theorem~\ref{thm:glest} can be replaced with $2\pi$. 
\end{remark}

A straightforward argument (see below) shows that if $F$ is a polynomial of degree $k$ in $d$ variables, then $\|F\|_{L^q(\mathbb{T}^d)} \gg \|F\|_{L^\infty(\mathbb{T}^d)}$ whenever $q\geq d \log{k}$. However, for Dirichlet polynomials we can do better, since the degree of the variables corresponding to large primes is restricted. We will see later that the exponent $\pi(N)$ in the following result is sharp. 
\begin{theorem}\label{thm:glest} 
	Let $F(z) = \sum_{n=1}^N a_n z(n)$. Set $d = \pi(N)$ and for $0<\lambda<1$ consider the set
	\[X_\lambda = \left\{z \in \mathbb{T}^d \,:\, |F(z)|\geq \lambda\|F\|_{L^\infty(\mathbb{T}^d)} \right\}.\]
	Then
	\[\mu_d(X_\lambda) \geq \left(\frac{1-\lambda}{\pi^2}\right)^{\pi(N)} e^{-\sqrt{N}}.\]
\end{theorem}
\begin{proof}
	Since $\mathbb{T}^d$ is compact, there is at least one point
	\[w = (e^{i\vartheta_1},\ldots, e^{i\vartheta_d})\]
	where the supremum is attained, so that $\|F\|_{L^\infty(\mathbb{T}^d)} = |F(w)|$. Let
	\[z = (e^{i\theta_1},\ldots, e^{i\theta_d})\]
	denote an arbitrary point on $\mathbb{T}^d$. Define $d_1 = \pi(\sqrt{N})$ and $d_2 = d - d_1$. It follows from the triangle inequality and Lemma~\ref{lem:bernstein} that 
	\begin{align*}
		|F(z) - F(w)| &\leq |F(z_1,z_2)-F(w_1,z_2)| + |F(w_1,z_2)-F(w_1,w_2)| \\
		&\leq \frac{\pi}{2} \|F\|_{L^\infty(\mathbb{T}^d)} \left(\frac{\log{N}}{\log{2}} \sup_{1 \leq j \leq d_1} |\theta_j - \vartheta_j| + \sup_{d_1 < j \leq d} |\theta_j - \vartheta_j| \right). 
	\end{align*}
	Here we used that if $z_2$ is fixed, then $F(\cdot,z_2)$ has degree at most
	\[\max_{1\leq n\leq N} \Omega(n) \leq \frac{\log{N}}{\log{2}}\]
	in $z_1$ and conversely if $z_1$ is fixed, then $F(z_1,\cdot)$ has degree $1$ in $z_2$. Moreover,
	\[|F(z)-F(w)|\geq |F(w)|-|F(z)| = \|F\|_{L^\infty(\mathbb{T}^d)} -|F(z)|.\]
	so in particular, $w \in X_\lambda$ whenever
	\[\lambda \leq 1 - \frac{\pi}{2}\left(\frac{\log{N}}{\log{2}} \sup_{1 \leq j \leq d_1} |\theta_j - \vartheta_j| + \sup_{d_1 < j \leq d} |\theta_j - \vartheta_j| \right).\]
	Hence $w \in X_\lambda$ also holds whenever
	\[\sup_{1 \leq j \leq d_1} |\theta_j - \vartheta_j| \leq \frac{1-\lambda}{\pi (\log{N})/\log{2}} \qquad \text{and} \qquad \sup_{d_1 < j \leq d} |\theta_j - \vartheta_j| \leq \frac{1-\lambda}{\pi}.\]
	Note that for any $\vartheta\in\mathbb{R}/[0,2\pi)$ we have that
	\[\mu_1 \left(\left\{z = e^{i\theta} \in \mathbb{T}\,:\, |\theta-\vartheta|\leq \xi \right\}\right) = \frac{\xi}{\pi},\]
	so we therefore conclude that
	\[\mu_d(X_\lambda) \geq \left(\frac{1}{\pi}\right)^d \left(\frac{1-\lambda}{\pi}\right)^{d_1+d_2} \left(\frac{\log{2}}{\log{N}}\right)^{d_1} \geq\left(\frac{1-\lambda}{\pi^2}\right)^{\pi(N)} e^{-\sqrt{N}},\]
	where we in the final inequality used that $d_1 = \pi(\sqrt{N})$. 
\end{proof}

Setting $\lambda = 1/2$ in Theorem~\ref{thm:glest}, we find that
\[\|F\|_{L^{2k}(\mathbb{T}^d)} \geq \left(\int_{X_{1/2}} |F(z)|^{2k}\,d\mu_d(z)\right)^\frac{1}{2k} \geq \frac{1}{2} \|F\|_{L^\infty(\mathbb{T}^d)} \left(\mu_d(X_{1/2})\right)^{1/2k}.\]
In particular, if $k\geq \pi(N)$ then $\|F\|_{L^{2k}(\mathbb{T}^d)} \geq (4\pi^2 e)^{-1} \|F\|_{L^\infty(\mathbb{T}^d)}$. By the Bohr correspondence, we have hence proved the following result. We will also present a different proof below, which was shown to us by K. Seip.
\begin{corollary}\label{cor:normcomp} 
	If $f(s) = \sum_{n=1}^N a_n n^{-s}$ and $k\gg N/\log{N}$, then
	\[\|f\|_\infty \ll \|f\|_{2k}.\]
\end{corollary}
\begin{proof}
	Let $\Psi(x,y)$ denote the number of integers less than $x$ whose prime factors are all less than $y$. The estimate 
	\begin{equation}\label{eq:Psiest} 
		\Psi(x,\log x) \ll e^{c \frac{\log x}{\log\log x}}, 
	\end{equation}
	can be found in \cite[pp.~270--271]{Granville02}. Now take $k$ to be an integer of size $N/\log N$ and consider the function $f^k$, where $f(s) = \sum_{n=1}^N a_n n^{-s}$. This function is a Dirichlet polynomial with at most $\Psi(N^k, N)$ nonzero terms, so the Cauchy--Schwarz inequality gives that
	\[\|f\|_\infty^k = \| f^k \|_\infty \le \sqrt{\Psi(N^k, N)} \| f^k\|_2 \ll e^{c\frac{N}{\log N}} \| f^k\|_2 = e^{c\frac{N}{\log N}}\|f\|_{2k}^k\]
	by \eqref{eq:Psiest} and hence $\| f \|_\infty \ll \| f\|_{2k}$ whenever $k\gg N/\log{N}\sim\pi(N)$. 
\end{proof}

To see that $k = N/\log{N}$ cannot generally be improved in Corollary~\ref{cor:normcomp} (and hence in Theorem~\ref{thm:glest}), we will use Khintchine's inequality (see \cite{KK01}). 
\begin{lemma}\label{lem:khintchine} 
	If $2\leq q < \infty$ and $d=\pi(N)$, then
	\[\left\|\sum_{p\leq N} a_p z(p)\right\|_{L^q(\mathbb{T}^d)} \leq {\Gamma\left(1+\frac{q}{2}\right)}^{1/q} \left(\sum_{p\leq N} |a_p|^2\right)^\frac{1}{2}.\]
\end{lemma}
In particular, we choose $a_p = 1$ for $p \leq N$, so that $F(z) = \sum_{p \leq N} z(p)$ and $\|F\|_{L^\infty(\mathbb{T}^d)}= \pi(N)$. By Lemma~\ref{lem:khintchine} with $q=2k$ we get that
\[\|F\|_{L^{2k}(\mathbb{T}^d)} \leq \Gamma(1+k)^{1/2k} \sqrt{\pi(N)} \leq\sqrt{k\pi(N)},\]
since $\Gamma(1+k) \leq k^k$. This shows that for general Dirichlet polynomials, the exponent $k = N/\log{N}$ in Corollary~\ref{cor:normcomp} cannot be improved. 

We will finally demonstrate that this can be substantially improved for the specific Dirichlet polynomial $F(z) = \sum_{n=1}^N z(n)/\sqrt{n}$. First, we recall the following well-known result.
\begin{lemma}\label{lem:yproj} 
	Let $F(z) = \sum_{n=1}^N a_n z(n)$, set $d=\pi(N)$ and for $y\geq2$ define
	\[S(N,y) = \{n \leq N \,:\, p|n \implies p\leq y\}.\]
	If $F_y(z) = \sum_{n \in S(N,y)} a_n z(n)$ then $\|F_y\|_{L^q(\mathbb{T}^{\pi(y)})} \leq \|F\|_{L^q(\mathbb{T}^d)}$ for $1 \leq q < \infty$. 
\end{lemma}
\begin{proof}
	It suffices to note that the map $F \mapsto F_y$ is defined by averaging out the variables corresponding to primes $p>y$ over the corresponding polytorus,
	\[F_y(z_1) = \int_{\mathbb{T}^{\pi(N)-\pi(y)}} F(z_1,z_2)\,d\mu_{\pi(N)-\pi(y)}(z_2)\]
	and then use H\"older's inequality in the inner integral. 
\end{proof}
\begin{theorem}\label{thm:zetacomp} 
	Set $F_N(z) = \sum_{n=1}^N z(n)/\sqrt{n}$ and fix $\varepsilon>0$. If $q\geq N^\varepsilon$, then $\|f_N\|_{L^q(\mathbb{T}^{\pi(N)})} \gg_\varepsilon \sqrt{N}$ as $N\to\infty$. 
\end{theorem}
\begin{proof}
	We first note that
	\[F_{N,y}(z) = \sum_{n \in S(N,y)} \frac{z(n)}{\sqrt{n}}\]
	is a polynomial in $\pi(y)$ variables of degree at most $\log{N}/\log{2}$. Using Lemma~\ref{lem:yproj} and then Lemma~\ref{lem:bernstein} as before, we may therefore conclude that
	\[\|F_N\|_{L^q(\mathbb{T}^{\pi(N})} \geq \|F_{N,y}\|_{L^q(\mathbb{T}^{\pi(y)})} \gg \|F_{N,y}\|_{L^\infty(\mathbb{T}^\pi(y))} = \sum_{n\in S(N,y)} \frac{1}{\sqrt{n}}\]
	provided $q \geq \pi(y)\log\left(\log{N}/\log{2}\right)$. Choosing $y = N^\varepsilon$ and using Abel summation and a standard estimate on smooth numbers (see e.g. \cite[Sec.~1]{Granville02}), we get
	\[\sum_{n\in S(N,N^\varepsilon)} \frac{1}{\sqrt{n}} \sim C_\varepsilon \sqrt{N}.\]
	Here $C_\varepsilon = 2 \varrho(1/\varepsilon)$ and $\varrho$ denotes Dickman's function. Hence we require that $q \geq \pi(N^\varepsilon) \log\log{N}$, so certainly $q\geq N^\varepsilon$ is acceptable. 
\end{proof}

\section*{Acknowledgements} The authors would like to express their gratitude to A. Bondarenko for several enlightening discussions and to the anonymous referee for a careful reading of the manuscript.

\bibliographystyle{amsplain} 
\bibliography{pi}

\end{document}